\documentclass[11pt]{amsart}
\usepackage[T1]{fontenc}
\usepackage{lmodern}
\usepackage{amsmath,amssymb,amsthm,mathtools}
\usepackage[a4paper,margin=30mm]{geometry}
\usepackage{microtype}
\usepackage[hidelinks]{hyperref}

\newtheorem{theorem}{Theorem}[section]
\newtheorem{lemma}[theorem]{Lemma}
\newtheorem{proposition}[theorem]{Proposition}
\newtheorem{corollary}[theorem]{Corollary}
\theoremstyle{remark}
\newtheorem{remark}[theorem]{Remark}

\newcommand{\Z}{\mathbb Z}
\newcommand{\Q}{\mathbb Q}
\newcommand{\eps}{\varepsilon}
\newcommand{\vtwo}{v_2}

\begin{document}

\title[Prime-power Diophantine tuples]{Prime-power Diophantine tuples}
\author{Andrej Dujella}

\dedicatory{Dedicated to the memory of my dear friend and coauthor Florian Luca}

\begin{abstract}
A positive $D(n)$-$m$-tuple is a set
$A=\{a_1,\ldots,a_m\}$ of distinct positive integers such that
$a_i a_j+n$ is a square for every $i\ne j$. 
In 2005, Dujella and Luca obtained an absolute bound for the cardinality of a
$D(p)$- or $D(-p)$-tuple of positive integers, uniformly in the prime $p$.
We extend the underlying gap argument to prime powers. 
An explicit toric elimination certificate, proved by
elementary linear algebra, replaces the unsaturated homogeneous elimination
step and is valid modulo every prime power. Explicit degree and height bounds
for this eliminant yield a uniform gap principle. Combined with the general
bound for bounded $|n|$, this shows that every positive reduced
$D(\pm p^r)$-tuple (i.e. tuple with elements not divisible by $p$) 
has less than $2^{121}$ elements, independently of $p$
and $r$. Thus, every positive $D(\pm p)$-tuple has at most $2^{121}$
elements, positive $D(\pm p^2)$-tuples have less than $2^{122}$ elements,
and the maximal cardinality of a positive $D(\pm p^r)$-tuple is $O(r)$
uniformly in $p$. In particular, this answers Problem~5.6 of~\cite{OpenProblems}.
\end{abstract}

\maketitle

\section{Introduction}

Let $n$ be a nonzero integer. A $D(n)$-$m$-tuple is a set
$A=\{a_1,\ldots,a_m\}$ of distinct nonzero integers such that
$a_i a_j+n$ is a square for every $i\ne j$. 
If elements of a $D(n)$-$m$-tuple are positive integers, 
we call such set a positive $D(n)$-$m$-tuple. 
Diophantus found a $D(256)$-quadruple $\{1, 33, 68, 105\}$
while the first $D(1)$-quadruple, the set $\{1, 3, 8, 120\}$, was found by Fermat 
(see Sections 1.1 and 1.2 of~\cite{Dbook}). 

Let $M_n^+$ denote the
supremum of the cardinalities of positive $D(n)$-tuples. 
By an elementary argument, Brown \cite{Brown} proved that 
$M_n^+=3$ if $n\equiv 2\pmod{4}$. In particular, $M_2^+=3$ and $M_{-2}^+=3$. Other known values of $M_n^+$ 
(obtained by much more involved methods, in particular, Baker's method on linear forms in 
logarithms of algebraic numbers) are: 
$M_1^+=4$ (by He, Togb\'e, and Ziegler \cite{HTZ}), 
$M_4^+=4$ (by Bliznac Trebje\v{s}anin and Filipin \cite{BTF}), 
and $M_{-1}^+=3$ (by Bonciocat, Cipu, Mignotte \cite{BCM}); 
their result also implies $M_{-4}^+=3$, since all elements of a $D(-4)$-quadruple are even. 

Dujella and Luca~\cite{DL} obtained the uniform estimate
$M_p^+,M_{-p}^+<3\cdot2^{168}$ for primes $p$, together with a bound
depending only on the number of prime divisors in the squarefree setting.
For general $n$, Yip~\cite{Yip} proved the asymptotic estimate
\[
 M_n^+\le (2+o(1))\log |n|\qquad (|n|\to\infty).
\]
In this paper, for prime powers we seek a bound independent of
the size of the prime, despite the repeated prime factor. Section~5.2.4
of~\cite{Dbook} gives a concise exposition of the gap-principle strategy for
prime parameters. The passage to prime powers requires concentration of the
full $p$-power in one factor, a separate $2$-adic analysis, and a replacement
for the homogeneous elimination step. The latter is supplied below by an
explicit toric certificate, which works on the locus where all tuple
coordinates are units modulo $p^r$. 
If
$n=\eps p^r$, where $p$ is prime and $\eps\in\{\pm1\}$, we call such a
$D(n)$-tuple \emph{reduced} if $p\nmid a$ for every $a\in A$.

Put
\[
 C_{\rm red}:=2^{121}.
\]

\begin{theorem}[reduced prime powers]\label{thm:reduced}
Let $p$ be a prime, $r\ge1$, and $\eps\in\{\pm1\}$. Every positive reduced
$D(\eps p^r)$-tuple has less than $C_{\rm red}$ elements. The bound is
independent of $p$, $r$, and $\eps$.
\end{theorem}

\begin{corollary}[prime parameters]\label{cor:prime}
Let $p$ be a prime and $\eps\in\{\pm1\}$. Every positive $D(\eps p)$-tuple
has at most $C_{\rm red}=2^{121}$ elements.
\end{corollary}

\begin{corollary}[prime squares]\label{cor:p2}
Let $p$ be a prime and $\eps\in\{\pm1\}$. Every positive
$D(\eps p^2)$-tuple has less than $C_{\rm red}+4<2^{122}$ elements. If
arbitrary nonzero integer elements are allowed, every $D(p^2)$-tuple has
less than $2C_{\rm red}+8<2^{123}$ elements, whereas every
$D(-p^2)$-tuple satisfies the same bound as in the positive case.
\end{corollary}

More generally, for $r\ge2$ we shall prove
\begin{equation}\label{eq:recursion-intro}
 M_{\eps p^r}^+\le C_{\rm red}+M_{\eps p^{r-2}}^+.
\end{equation}
Therefore, 
\[
 M_{\eps p^{2s}}^+\le sC_{\rm red}+4,
 \qquad
 M_{\eps p^{2s+1}}^+\le (s+1)C_{\rm red},
\]
and hence $M_{\pm p^r}^+=O(r)$ with an absolute effective implied constant,
uniformly in $p$.

\section{Preliminary lower bounds}

Write $N=p^r$. The following elementary reduction will be used repeatedly.

\begin{lemma}
\label{lem:lower}
Let $A=\{a_1<a_2<\cdots\}$ be a positive reduced $D(\eps N)$-tuple.
If $\eps=1$, then $a_3>N^{1/4}$, so at most two elements fail the lower
bound $a>N^{1/4}$. If $\eps=-1$, then $a_2>N^{1/2}$, so at most one
element fails the stronger lower bound $a>N^{1/2}$.
\end{lemma}

\begin{proof}
Suppose first that $\eps=1$, and write
\[
 a_1a_3+N=x_{13}^2,\qquad a_2a_3+N=x_{23}^2.
\]
Then $x_{23}>x_{13}>\sqrt N$, and therefore
\[
 a_3(a_2-a_1)=x_{23}^2-x_{13}^2
 =(x_{23}-x_{13})(x_{23}+x_{13})>2\sqrt N.
\]
Since $a_2-a_1<a_3$, this gives $a_3>N^{1/4}$.

Now let $\eps=-1$. Reducedness rules out $x_{12}=0$ in
$a_1a_2-N=x_{12}^2$, because $a_1a_2=N$ would force a factor $p$ in one
of $a_1,a_2$. Thus
\[
 a_1a_2=N+x_{12}^2>N,
\]
and hence $a_2^2>a_1a_2>N$.
\end{proof}

\section{The prime-power factorization}

Write
\[
 N=p^r,\qquad n=\eps N.
\]
Let $a<b<c$ be elements of a reduced $D(\eps N)$-tuple and write
\[
 ab+\eps N=x^2,\qquad ac+\eps N=z^2,\qquad bc+\eps N=y^2,
\]
with $x,y,z\ge0$.  Put
\[
 P=abc,\qquad S=a+b+c,
\]
and define, as in Section~5.2.4 of~\cite{Dbook},
\begin{equation}
 e=\eps NS+2P-2xyz,
 \qquad
 \bar e=\eps NS+2P+2xyz.
\end{equation}
A direct computation gives
\begin{equation}\label{eq:factor}
 e\bar e
 =N^2(c-a-b+2x)(c-a-b-2x).
\end{equation}
The common value in~\eqref{eq:factor} is zero exactly in the regular cases
\begin{equation}
 c=a+b\pm2x.
\end{equation}
For $D(N)$ only the plus sign can give $c>b$, while both signs may occur for
$D(-N)$.

\begin{lemma}
\label{lem:local}
Assume that the triple is not regular.
\begin{enumerate}
\item If $p$ is odd, then exactly one of $e,\bar e$ is divisible by $N^2$.
\item If $p=2$ and $r\ge3$, then
\[
 \{\vtwo(e),\vtwo(\bar e)\}=\{2,2r-2\}.
\]
In particular, exactly one of $e,\bar e$ is divisible by $N^2/4$.
\end{enumerate}
\end{lemma}

\begin{proof}
If $p$ is odd, then $\bar e-e=4xyz$.  Since the tuple is reduced,
$p\nmid xyz$, and hence $e$ and $\bar e$ cannot both be divisible by $p$.
Equation~\eqref{eq:factor} implies that the whole factor $N^2=p^{2r}$
divides exactly one of them.

Let $p=2$ and $r\ge3$.  The elements $a,b,c$ and the square roots $x,y,z$
are odd.  Hence $4\mid e$ and $4\mid\bar e$, while
\[
 \vtwo(\bar e-e)=\vtwo(4xyz)=2.
\]
Moreover $c-a-b$ and $x$ are odd, so $(c-a-b)^2-4x^2$ is odd.  Taking
$2$-adic valuations in~\eqref{eq:factor} gives
$\vtwo(e)+\vtwo(\bar e)=2r$.  The smaller valuation is exactly $2$, which
proves the statement.
\end{proof}

Put
\[
 \eta=\begin{cases}
 1,&p\text{ odd},\\
 4,&p=2.
 \end{cases}
\]
For $p=2$ this notation will only be used when $r\ge3$. 
Note that the cases $N=2,4$ are solved by the results mentioned in Introduction. 
For a nonregular triple, we
say that it lies in the \emph{$e$-branch} if $N^2/\eta\mid e$, and in the
\emph{$\bar e$-branch} otherwise. By Lemma~\ref{lem:local}, these are the
only two possibilities and exactly one occurs.

\section{The easy branch}

Let
\[
 \mathcal R=(c-a-b)^2-4x^2
 =a^2+b^2+c^2-2ab-2ac-2bc-4\eps N.
\]
Then $e\bar e=N^2\mathcal R$, and, since $a,b<c$,
\begin{equation}\label{eq:Rbound}
 |\mathcal R|<9c^2+4N.
\end{equation}

\begin{lemma}
\label{lem:easy}
Assume that $N>2^{2^{117}}$, the triple is nonregular, and
$N^2/\eta\mid e$.  Assume also $a>N^{1/4}$ if $\eps=1$, and
$a>N^{1/2}$ if $\eps=-1$.  Then
\[
 c>N^{1/8}b.
\]
\end{lemma}

\begin{proof}
Suppose first that $\eps=1$.  Since $xyz>P$, we have $\bar e>4P$.  From
$|e|\ge N^2/\eta$,~\eqref{eq:factor}, and~\eqref{eq:Rbound},
\[
 4abc<\bar e\le \eta |\mathcal R|<\eta(9c^2+4N).
\]
If $c^2\le N$, then the right-hand side is $<13\eta N$, while
$\bar e>NS>3N^{5/4}$, a contradiction.  Hence $c^2>N$, and
$4abc<13\eta c^2$.  Therefore
\[
 c>\frac{4ab}{13\eta}>N^{1/8}b.
\]

Now let $\eps=-1$. Since $ab,ac,bc>N$, we have
$x<\sqrt{ab}$, $z<\sqrt{ac}$ and $y<\sqrt{bc}$, and hence $xyz<P$.
Put $\Delta_0=P-xyz>0$. We have
\[
 P^2-(xyz)^2=NPS-N^2(ab+ac+bc)+N^3<NPS,
\]
so $\Delta_0<NS$ and
\begin{equation}
 |e|=|2\Delta_0-NS|<NS.
\end{equation}
The divisibility assumption gives $S>N/\eta$, hence
\begin{equation}\label{eq:c-N}
 c>\frac{N}{3\eta}.
\end{equation}
If $b\le N^{7/8}/(3\eta)$, this already gives $c>N^{1/8}b$.

Assume $b>N^{7/8}/(3\eta)$.  Since $a>N^{1/2}$, we have 
$ab>N^{11/8}/(3\eta)$.  Put
\[
 u=\frac{N}{ab},\qquad v=\frac{N}{ac},\qquad w=\frac{N}{bc},
 \qquad \rho=u+v+w.
\]
Then $\rho<9\eta N^{-3/8}$ and
\[
 \frac{xyz}{P}=\sqrt{(1-u)(1-v)(1-w)}\ge1-\rho.
\]
Since $NS/P=\rho$, we have 
\[
 \frac{\bar e}{P}=-\rho+2+2\frac{xyz}{P}\ge4-3\rho>1.
\]
Thus $\bar e>P$.  By~\eqref{eq:c-N}, $c^2>N$, so
\[
 abc<\bar e\le\eta|\mathcal R|<13\eta c^2.
\]
Therefore, 
\[
 c>\frac{ab}{13\eta}>\frac{N^{1/2}}{13\eta}b>N^{1/8}b.
\]
\end{proof}

\section{The hard branch and full-modulus congruences}

We now suppose that the large prime-power divisor occurs in $\bar e$ for all
triples under consideration.  For odd $p$, $N^2\mid\bar e$ gives
\begin{equation}
 xyz\equiv-abc\pmod N.
\end{equation}
For $p=2$, $r\ge3$, Lemma~\ref{lem:local} gives
$\vtwo(\bar e)=2r-2$.  Put $h=N/2$.  Since $\bar e/2$ is divisible by $N$,
\[
 xyz+abc+\eps h(a+b+c)\equiv0\pmod N.
\]
The sum $a+b+c$ is odd and $h\equiv-h\pmod N$, so
\begin{equation}\label{eq:triple-two}
 xyz\equiv-abc+h\pmod N.
\end{equation}
The term $h$ disappears when two such triple congruences are multiplied.

\begin{lemma}
\label{lem:opposite}
Let $a<b<c<d$ belong to a reduced $D(\eps N)$-tuple and write
\[
\begin{array}{lll}
 ab+\eps N=x_1^2,&ac+\eps N=x_2^2,&bc+\eps N=x_3^2,\\
 ad+\eps N=x_4^2,&bd+\eps N=x_5^2,&cd+\eps N=x_6^2.
\end{array}
\]
Assume that every triple among $a,b,c,d$ is in the $\bar e$-branch.  Then
\begin{equation}\label{eq:opposite}
 x_1x_6\equiv x_2x_5\equiv x_3x_4\pmod N.
\end{equation}
\end{lemma}

\begin{proof}
For odd $p$ this is Lemma~3.1 of~\cite{DL}, with the same proof in
$\Z/N\Z$. All elements being cancelled are units modulo $N$.

For $p=2$, multiply~\eqref{eq:triple-two} for $(a,b,c)$ and $(a,b,d)$.
The right-hand side is
\[
 (-abc+h)(-abd+h)
 =a^2b^2cd-hab(c+d)+h^2\equiv a^2b^2cd\pmod N,
\]
because $c+d$ is even and $h^2\equiv0\pmod N$.  Cancelling the unit
$x_1^2\equiv ab\pmod N$ gives
$x_2x_3x_4x_5\equiv abcd\pmod N$.  The triples $(a,b,c)$ and $(a,c,d)$
give $x_1x_3x_4x_6\equiv abcd\pmod N$.  Cancelling $x_3x_4$ yields the
first congruence in~\eqref{eq:opposite}. The rest follows by permutation.
\end{proof}

Define positive integers
\begin{equation}\label{eq:lambdas}
\begin{split}
 \lambda_1&=\frac{\eps(x_1x_6-x_2x_5)}{N},\\
 \lambda_2&=\frac{\eps(x_1x_6-x_3x_4)}{N},\\
 \lambda_3&=\frac{\eps(x_2x_5-x_3x_4)}{N}.
\end{split}
\end{equation}
The exact identities are
\begin{equation}\label{eq:exact-lambdas}
\begin{split}
 (d-a)(c-b)&=\lambda_1(x_1x_6+x_2x_5),\\
 (d-b)(c-a)&=\lambda_2(x_1x_6+x_3x_4),\\
 (d-c)(b-a)&=\lambda_3(x_2x_5+x_3x_4).
\end{split}
\end{equation}
If $\eps=1$, then, for example,
\[
 d^2>(d-a)(c-b)
 =\lambda_1(x_1x_6+x_2x_5)
 >2\lambda_1\sqrt{abcd}>\lambda_1a^2.
\]
Hence
\begin{equation}\label{eq:lambda-positive-gap}
 d>\sqrt{\lambda_i}\,a\qquad(i=1,2,3),
\end{equation}
and the arguments for $\lambda_2$ and $\lambda_3$ are identical.
For the negative sign one needs the following observation.

\begin{lemma}
\label{lem:negative-lambda}
Assume $\eps=-1$, every one of the four triples contained in
$\{a,b,c,d\}$ lies in the $\bar e$-branch, and $N>2^{2^{117}}$.
Then either
\[
 c>N^{1/16}a,
\]
or~\eqref{eq:lambda-positive-gap} holds for $i=1,2,3$.
\end{lemma}

\begin{proof}
Apply the $\bar e$-branch to the sorted triple $a<b<c$. If $\bar e<0$, then
$|\bar e|\ge N^2/\eta$ gives $S>N/\eta$. Also
$2(P+xyz)<NS<3Nc$, hence $ab<3N/2$, and therefore
\[
 \frac ca>\frac{N/(3\eta)}{\sqrt{3N/2}}>N^{1/16}.
\]

Suppose $\bar e>0$. Since $\bar e<4P$ and
$\bar e\ge N^2/\eta$,
\[
 abc>\frac{N^2}{4\eta}.
\]
If $ab\le N^{5/4}$, then $c>N^{3/4}/(4\eta)$ and
$a\le N^{5/8}$, so again $c>N^{1/16}a$.

It remains to consider $ab>N^{5/4}$. Every pair $uv$ among $a,b,c,d$
then satisfies $uv>N^{5/4}$ and
\[
 \sqrt{uv-N}>\sqrt{uv}\sqrt{1-N^{-1/4}}.
\]
Hence
\[
 x_1x_6+x_2x_5
 >2(1-N^{-1/4})\sqrt{abcd}>\sqrt{abcd}.
\]
The first identity in~\eqref{eq:exact-lambdas} gives
\[
 d^2>(d-a)(c-b)>\lambda_1\sqrt{abcd}>\lambda_1a^2.
\]
The arguments for $\lambda_2$ and $\lambda_3$ are identical.
\end{proof}

We shall also need the polynomial congruence attached to a quadruple. By
\eqref{eq:opposite} and~\eqref{eq:exact-lambdas},
\[
 (d-a)^2(c-b)^2
 \equiv4\lambda_1^2(x_1x_6)^2
 \equiv4\lambda_1^2abcd\pmod N.
\]
The analogous congruences for $\lambda_2$ and $\lambda_3$, multiplied
together, give
\begin{equation}\label{eq:quadruple-congruence}
 \Delta(a,b,c,d)^2\equiv64\Lambda(abcd)^3\pmod N,
 \qquad
 \Lambda=(\lambda_1\lambda_2\lambda_3)^2,
\end{equation}
where
\[
 \Delta(a,b,c,d)=
 (d-c)(d-b)(d-a)(c-b)(c-a)(b-a).
\]

\section{A toric elimination lemma}

The five homogeneous equations in the elimination step have fixed coordinate
base points. We therefore work explicitly on the torus, where every tuple
coordinate is a unit modulo $N$. This gives the congruence modulo $N$
without any division by an integer content.

Let $b_1,\ldots,b_5$ be a \emph{labelled} quintuple of pairwise distinct
elements. For each $i$, arrange the four elements $b_j$, $j\ne i$,
increasingly before defining the three positive parameters
$\lambda_{1,i},\lambda_{2,i},\lambda_{3,i}$ by~\eqref{eq:lambdas}, and put
\[
 \Lambda_i=(\lambda_{1,i}\lambda_{2,i}\lambda_{3,i})^2,
 \qquad T_i=64\Lambda_i.
\]
Both $\Lambda_i$ and the square of the Vandermonde product of the four
retained elements are unchanged by a permutation of those four elements.
Thus~\eqref{eq:quadruple-congruence} gives
\begin{equation}\label{eq:homogPi}
 \Delta_i(b)^2\equiv
 T_i\left(\prod_{j\ne i}b_j\right)^3\pmod N,
\end{equation}
where
\[
 \Delta_i(X)=\prod_{\substack{j<k\\j,k\ne i}}(X_k-X_j).
\]

All $b_j$ are units modulo $N$. Dehomogenize at the labelled coordinate
$b_5$ by putting, in $\Z/N\Z$,
\[
 t_j=b_jb_5^{-1}\quad(1\le j\le4),\qquad t_5=1.
\]
After division of~\eqref{eq:homogPi} by $b_5^{12}$, we obtain
\begin{equation}\label{eq:AiBi}
 A_i(t)\equiv T_iB_i(t)\pmod N,
\end{equation}
where
\[
 A_i(t)=\Delta_i(t)^2,
 \qquad
 B_i(t)=\left(\prod_{j\ne i}t_j\right)^3.
\]
Here $\deg A_i,\deg B_i\le12$ and every $B_i(t)$ is a unit modulo $N$.
For an integral polynomial $F$, let $h(F)$ denote the sum of the absolute
values of its coefficients.

\begin{lemma}[toric eliminant]\label{lem:toric}
Let
\[
 D_0=24^4=331776.
\]
There exists a nonconstant polynomial
\[
 Q(T_1,\ldots,T_5)\in\Z[T_1,\ldots,T_5]
\]
with
\begin{equation}\label{eq:Q-degree}
 1\le q:=\deg Q\le5D_0<2^{21},
\end{equation}
and polynomials
\[
 G_i\in\Z[t_1,\ldots,t_4,T_1,\ldots,T_5]\qquad(1\le i\le5)
\]
such that
\begin{equation}\label{eq:toric-certificate}
 (t_1t_2t_3t_4)^{12D_0}Q(T)
 =\sum_{i=1}^5G_i(t,T)\bigl(A_i(t)-T_iB_i(t)\bigr).
\end{equation}
The same polynomial $Q$ may be chosen so that
\begin{equation}\label{eq:Q-height-combined}
 \log_2 \bigl(h(Q)64^q\bigr)<2^{116}.
\end{equation}
Consequently every reduced labelled quintuple satisfying~\eqref{eq:AiBi}
satisfies
\begin{equation}\label{eq:Q-cong}
 Q(64\Lambda_1,\ldots,64\Lambda_5)\equiv0\pmod N.
\end{equation}
\end{lemma}

\begin{proof}
For $\alpha=(\alpha_1,\ldots,\alpha_5)$ with
$0\le\alpha_i\le D_0$, define
\[
 F_\alpha(t)=\prod_{i=1}^5A_i(t)^{\alpha_i}B_i(t)^{D_0-\alpha_i}.
\]
For a fixed variable $t_j$, only the four factors indexed by $i\ne j$
involve $t_j$. In each such factor, the $t_j$-degree is at most
\[
 6\alpha_i+3(D_0-\alpha_i)\le6D_0.
\]
Hence
\[
 \deg_{t_j}F_\alpha\le24D_0\qquad(j=1,\ldots,4).
\]
Therefore, the coefficient vectors of all $F_\alpha$ lie in a vector space of
dimension at most
\[
 M=(24D_0+1)^4.
\]
There are $(D_0+1)^5$ such polynomials. Since $D_0=24^4$,
\[
 (24D_0+1)^4<[24(D_0+1)]^4
 =D_0(D_0+1)^4<(D_0+1)^5.
\]
Thus the family is linearly dependent.

Let $\rho$ be the rank of its integer coefficient matrix. Choose $\rho$
linearly independent columns and one further column in their span. Choose a
nonzero $\rho\times\rho$ minor of the independent columns and use the
corresponding signed maximal minors to obtain an integral relation
\begin{equation}\label{eq:F-relation}
 \sum_\alpha c_\alpha F_\alpha(t)=0
\end{equation}
supported on at most $\rho+1$ columns. Using these same coefficients, define
\[
 Q(T)=\sum_\alpha c_\alpha T_1^{\alpha_1}\cdots T_5^{\alpha_5}.
\]
The monomials $T^\alpha$ are distinct, so $Q\ne0$. Moreover $Q$ is
nonconstant: a constant relation would be supported only at
$\alpha=(0,\ldots,0)$, whereas
\[
 F_{\mathbf0}=\prod_{i=1}^5B_i^{D_0}\ne0.
\]
This proves~\eqref{eq:Q-degree}.

Put $S(t)=\prod_iB_i(t)^{D_0}$. Since
\[
 \prod_{i=1}^5B_i(t)=(t_1t_2t_3t_4)^{12},
\]
we have $S=(t_1t_2t_3t_4)^{12D_0}$. For each $\alpha$,
$F_\alpha-ST^\alpha$ belongs to the ideal generated by the five polynomials
$A_i-T_iB_i$. Indeed, this follows by replacing the factors one at a time and
using $U^m-V^m=(U-V)(U^{m-1}+\cdots+V^{m-1})$. Summing with the coefficients
$c_\alpha$ in~\eqref{eq:F-relation} gives~\eqref{eq:toric-certificate}.
Since $t_1t_2t_3t_4$ is a unit modulo $N$,~\eqref{eq:Q-cong} follows.

It remains to bound the same maximal-minor relation. Each $\Delta_i$ is a
product of six binomials, so
\[
 h(A_i)\le2^{12},\qquad h(B_i)=1,
\]
and therefore $h(F_\alpha)\le2^{60D_0}$. The Euclidean norm of every
coefficient column is at most its $\ell^1$-norm, hence at most $2^{60D_0}$.
Hadamard's inequality gives
\[
 h(Q)\le(\rho+1)2^{60D_0\rho}.
\]
Furthermore,
\[
 \rho\le(24D_0+1)^4<2^{92},\qquad q\le5D_0.
\]
Hence,
\[
 \begin{split}
 \log_2 \bigl(h(Q)64^q\bigr)
 &\le \log_2(\rho+1)+60D_0\rho+6q\\
 &<92+60D_0(24D_0+1)^4+30D_0\\
 &<2^{116},
 \end{split}
\]
where the last inequality is a direct integer calculation. This proves
\eqref{eq:Q-height-combined}.
\end{proof}

\section{From the toric eliminant to a uniform gap}

Assume now
\begin{equation}\label{eq:Nlarge}
 N>2^{2^{117}}.
\end{equation}
Let $q=\deg Q$ and put
\[
 \delta=\frac1{12q},\qquad \gamma=\frac\delta2=\frac1{24q}.
\]
Since $q<2^{21}$,
\begin{equation}
 \gamma>2^{-26}.
\end{equation}

Consider a labelled quintuple of pairwise distinct elements for which none of
the ten determined triples lies in the $e$-branch. Thus all ten triples lie
in the $\bar e$-branch, and each of the five four-element subsets satisfies
Lemma~\ref{lem:opposite}. If one of the fifteen parameters
$\lambda_{j,i}$ is at least $N^\delta$, then after arranging the relevant
quadruple increasingly,~\eqref{eq:lambda-positive-gap} gives a gap of
exponent $\gamma$ when $\eps=1$. When $\eps=-1$,
Lemma~\ref{lem:negative-lambda} gives either the same conclusion or the
stronger $N^{1/16}$-gap. We may therefore suppose that all fifteen parameters
satisfy
\[
 \lambda_{j,i}<N^\delta.
\]
Then $\Lambda_i<N^{6\delta}$, and Lemma~\ref{lem:toric} gives
\[
 Q(64\Lambda_1,\ldots,64\Lambda_5)\equiv0\pmod N.
\]
Moreover,
\[
 \begin{split}
 |Q(64\Lambda_1,\ldots,64\Lambda_5)|
 &\le h(Q)(64N^{6\delta})^q\\
 &=h(Q)64^qN^{1/2}.
 \end{split}
\]
By~\eqref{eq:Q-height-combined} and~\eqref{eq:Nlarge},
\[
 h(Q)64^q<2^{2^{116}}<N^{1/2}.
\]
Hence
\begin{equation}\label{eq:Q-zero}
 Q(64\Lambda_1,\ldots,64\Lambda_5)=0.
\end{equation}

We now convert~\eqref{eq:Q-zero} into a polynomial condition in the five
normalized tuple elements. For real variables in a region where the radicand
is positive, write
\[
 s_\eps(u,v)=\sqrt{uv+\eps}.
\]
For a quadruple $(x,y,z,t)$ set
\[
 U=s_\eps(x,y)s_\eps(z,t),\qquad
 V=s_\eps(x,z)s_\eps(y,t),\qquad
 W=s_\eps(x,t)s_\eps(y,z),
\]
and define
\begin{equation}
 \mathcal L_\eps(x,y,z,t)=((U-V)(U-W)(V-W))^2.
\end{equation}
The function $\mathcal L_\eps$ is symmetric in its four arguments, because
a permutation of the four vertices permutes the three partitions
represented by $U,V,W$.

For $z=(z_1,\ldots,z_5)$ define $\Phi_\eps$ componentwise by
\begin{equation}
 (\Phi_\eps(z))_i=
 \mathcal L_\eps(z_1,\ldots,\widehat z_i,\ldots,z_5),
 \qquad 1\le i\le5,
\end{equation}
where the four displayed arguments are inserted in their inherited order.
The symmetry makes that order irrelevant. For a normalized tuple
$z_i=b_i/\sqrt N$, the $i$-th component is precisely~$\Lambda_i$.

\begin{lemma}
\label{lem:jacobian}
For each $\eps\in\{\pm1\}$,
\[
 \det D\Phi_\eps(2,3,4,5,6)\ne0.
\]
Consequently, for every nonzero polynomial $Q\in\Z[T_1,\ldots,T_5]$,
the real-analytic function $Q(64\Phi_\eps(z))$ is not identically zero.
\end{lemma}

\begin{proof}
Exact rational interval arithmetic gives
\[
 \det D\Phi_{+}(2,3,4,5,6)
 \in[-1.22625271490038,-1.22625271490037]\cdot10^{-31},
\]
\[
 \det D\Phi_{-}(2,3,4,5,6)
 \in[2.9335784988635907,2.9335784988635908]\cdot10^{-29}.
\]
Both intervals exclude zero. The deterministic certificate is described in
Appendix~\ref{app:jacobian}.

On a neighbourhood of $(2,3,4,5,6)$ all ten radicands are positive, so
$\Phi_\eps$ is real analytic. Since its Jacobian determinant is nonzero at
that point, the inverse function theorem shows that its image contains a
nonempty open subset of $\mathbb R^5$. A nonzero polynomial cannot vanish on
a nonempty open subset of $\mathbb R^5$. Therefore
$Q(64\Phi_\eps(z))$ is not identically zero.
\end{proof}

Let
\[
 K=\Q(z_1,\ldots,z_5),\qquad
 L=K\bigl(s_{ij}:1\le i<j\le5\bigr),
 \qquad s_{ij}^2=z_iz_j+\eps.
\]
The extension $L/K$ is multiquadratic, hence Galois, and
$[L:K]\le2^{10}$. Put
\[
 G_\eps(z)=Q(64\Phi_\eps(z))\in L.
\]
By Lemma~\ref{lem:jacobian}, $G_\eps\ne0$. It is integral over
$\Z[z_1,\ldots,z_5]$, since every $s_{ij}$ is integral over this ring.
Define
\begin{equation}\label{eq:norm}
 R_{\eps,0}(z)
 =\operatorname{Norm}_{L/K}(G_\eps)
 =\prod_{\sigma\in\operatorname{Gal}(L/K)}\sigma(G_\eps).
\end{equation}
This norm belongs to $K$ and is integral over
$\Z[z_1,\ldots,z_5]$. Since the latter ring is integrally closed,
\[
 R_{\eps,0}\in\Z[z_1,\ldots,z_5].
\]
It is nonzero because $G_\eps\ne0$. The identity automorphism occurs in the
product~\eqref{eq:norm}. Therefore, after any specialization at which the
radicals are defined,
\begin{equation}\label{eq:R0imp}
 G_\eps(z)=0\quad\Longrightarrow\quad R_{\eps,0}(z)=0.
\end{equation}
Each conjugate of $G_\eps$ is obtained by changing some signs of the radicals
and has weighted degree at most $12q$, when both the variables $z_i$ and the
radicals $s_{ij}$ are assigned weight~$1$. Thus
\begin{equation}
 \deg R_{\eps,0}\le12q\,[L:K]\le12q\,2^{10}.
\end{equation}

Let us put
\[
 \mathcal U_\eps(x,y,z)
 =x^2+y^2+z^2-2xy-2xz-2yz-4\eps.
\]
This symmetric polynomial vanishes exactly when the corresponding normalized
triple is regular. Multiply $R_{\eps,0}$ by the ten diagonal factors
$z_i-z_j$ and by the ten factors
$\mathcal U_\eps(z_i,z_j,z_k)$, $1\le i<j<k\le5$. The resulting nonzero
polynomial $R_\eps$ has
\begin{equation}
 \ell:=\deg R_\eps\le12q\,2^{10}+30<2^{35}.
\end{equation}

We use the elementary polynomial nonvanishing lemma: if a nonzero polynomial
in $s$ variables has total degree $d$ and a finite set $\mathcal M$ has more
than $d$ elements, then the polynomial does not vanish on all of
$\mathcal M^s$ (compare Lemma~3.2 of~\cite{DL}).

\begin{proposition}
\label{prop:gap}
Assume that $N>2^{2^{117}}$. After the lower-bound reduction of
Lemma~\ref{lem:lower}, the ordered elements
$a_1<a_2<\cdots<a_m$ of a positive reduced $D(\eps N)$-tuple satisfy
\begin{equation}\label{eq:gap}
 a_{i+\ell}>N^\gamma a_i
 \qquad(1\le i\le m-\ell),
\end{equation}
where $\ell<2^{35}$ and $\gamma>2^{-26}$.
\end{proposition}

\begin{proof}
Fix $i$ and apply the nonvanishing lemma to
\[
 \mathcal M=
 \left\{\frac{a_i}{\sqrt N},\ldots,
 \frac{a_{i+\ell}}{\sqrt N}\right\}.
\]
It gives a labelled point $(z_1,\ldots,z_5)\in\mathcal M^5$ such that
$R_\eps(z_1,\ldots,z_5)\ne0$. The diagonal factors make these five entries
pairwise distinct, and the symmetric regular factors ensure that none of the
ten triples determined by them is regular.

If one of those ten triples lies in the $e$-branch, arrange that triple
increasingly and apply Lemma~\ref{lem:easy}. The resulting gap between two
selected elements implies the required gap between the endpoints
$a_i,a_{i+\ell}$ of the ambient block. Suppose therefore that none of the
ten triples lies in the $e$-branch. Then all ten lie in the $\bar e$-branch.
For each of the five four-element subsets, arrange its elements increasingly
before applying Lemmas~\ref{lem:opposite} and~\ref{lem:negative-lambda}.
If $\eps=-1$, the latter lemma may already give a gap of exponent $1/16$.
If some $\lambda_{j,i}\ge N^\delta$, then
\eqref{eq:lambda-positive-gap} or Lemma~\ref{lem:negative-lambda} gives a
gap of exponent $\gamma=\delta/2$ between two selected elements, hence
between the block endpoints.

Finally suppose that all fifteen parameters are less than $N^\delta$.
The labelled toric construction applies without reordering the five selected
coordinates. Equation~\eqref{eq:Q-zero} gives
$G_\eps(z)=Q(64\Phi_\eps(z))=0$, and~\eqref{eq:R0imp} implies
$R_{\eps,0}(z)=0$. This contradicts the chosen nonzero value of $R_\eps$.
Thus~\eqref{eq:gap} holds.
\end{proof}

\section{Proof of the main theorem and consequences}

\begin{proof}[Proof of Theorem~\ref{thm:reduced}]
First suppose $N\le2^{2^{117}}$. It was proved in~\cite{DujellaBounds} that
\[
 M_n^+\le31\quad(|n|\le400),
 \qquad
 M_n^+<15.476\log|n|\quad(|n|>400),
\]
where $\log$ denotes the natural logarithm. Thus, in the present
range,
\[
 M_{\eps N}^+<15.476\,2^{117}\log2
 <11\cdot2^{117}<2^{121}
\]
when $N>400$, while the case $N\le400$ is immediate.

Assume therefore that $N > 2^{2^{117}}$. By Lemma~\ref{lem:lower}, after deleting
at most two elements in the positive case, and at most one in the negative
case, the lower bound $a>N^{1/4}$ holds for every remaining element. By the
large-element estimate in~\cite{Dcamb} (see also~\cite[Proposition~5.2.1]{Dbook}), at most $21$
elements exceed $N^3$. Hence, after deleting at most $23$ elements, we may
assume
\[
 N^{1/4}<a\le N^3
\]
for every remaining element.

Let $m_{\rm rem}$ be the number of remaining elements. Iterating
Proposition~\ref{prop:gap}, gives
\[
 m_{\rm rem}<\ell\left(1+\frac{11}{4\gamma}\right).
\]
Since $\ell<2^{35}$ and $\gamma=1/(24q)$ with $q<2^{21}$, we have
\[
 \frac{11}{4\gamma}=66q<2^{28}.
\]
Therefore $m_{\rm rem}<2^{64}$, and restoring the deleted elements gives
$m<2^{65}<2^{121}$.
\end{proof}

\begin{proof}[Proof of Corollary~\ref{cor:prime}]
A positive $D(\eps p)$-tuple contains at most one element divisible by $p$.
Indeed, if $a=p\alpha$ and $b=p\beta$ were two such elements, then
\[
 ab+\eps p=p(p\alpha\beta+\eps)
\]
would have $p$-adic valuation exactly $1$, and hence could not be a square.
The remaining subtuple is reduced. By Theorem~\ref{thm:reduced} it contains
at most $C_{\rm red}-1$ elements, so the whole tuple has at most
$C_{\rm red}=2^{121}$ elements.
\end{proof}

We next prove recurrence~\eqref{eq:recursion-intro}. Split a positive
$D(\eps p^r)$-tuple into the elements divisible by $p$ and those coprime to
$p$. The latter subtuple is reduced. If $a=p\alpha$ and $b=p\beta$ belong
to the divisible subtuple, then
\[
 p^2\alpha\beta+\eps p^r=x^2.
\]
Thus $p\mid x$, and division by $p^2$ gives
\[
 \alpha\beta+\eps p^{r-2}=(x/p)^2.
\]
Dividing every element of the divisible subtuple by $p$ therefore produces
a $D(\eps p^{r-2})$-tuple, which proves~\eqref{eq:recursion-intro}.

For $r=2$, the terminal tuple has property $D(1)$ or $D(-1)$. A positive
$D(1)$-tuple has at most four elements by the result of He, Togb\'e and Ziegler~\cite{HTZ},
and a positive $D(-1)$-tuple has at most three elements by the result of
Bonciocat, Cipu and Mignotte~\cite{BCM}. In particular, for either sign,
\[
 M_{\eps p^2}^+<C_{\rm red}+4<2^{122}.
\]
This proves the positive part of Corollary~\ref{cor:p2}.

If arbitrary nonzero integer elements are allowed and $n=p^2>0$, split the
tuple into its positive and negative elements. The positive elements and the
absolute values of the negative elements are each positive $D(p^2)$-tuples.
Hence the whole tuple has fewer than
\[
 2C_{\rm red}+8<2^{123}
\]
elements. If $n=-p^2$, mixed signs are impossible, because the product of
oppositely signed elements minus $p^2$ is negative. After changing all signs
if necessary, the tuple is positive, so the same positive bound applies.

Finally, iterating~\eqref{eq:recursion-intro} and using
Corollary~\ref{cor:prime} gives
\[
 M_{\eps p^{2s}}^+\le sC_{\rm red}+4,
 \qquad
 M_{\eps p^{2s+1}}^+\le (s+1)C_{\rm red}.
\]
Thus $M_{\pm p^r}^+=O(r)$ uniformly in $p$, with an explicit linear bound.

\begin{remark}
The five homogeneous equations used in the elimination step have fixed
coordinate base points. In particular, every coordinate point of
$\mathbb P^4$ is a common zero for every value of the five parameters.
Hence an unsaturated homogeneous resultant cannot by itself produce the
required nonzero eliminant. Lemma~\ref{lem:toric} instead works on the torus
and supplies an integral monomial-saturated certificate, whose monomial
factor is invertible for reduced tuples modulo $p^r$.
\end{remark}

\begin{remark}
For completeness, Lemma~\ref{lem:easy} gives an independent proof of the
negative-sign easy branch. This avoids relying on the displayed lower-bound
estimate following equation~(4) of~\cite{DL}, where the comparison between
$xyz$ and $abc$ is not in the direction used in the printed inequality.
The replacement is valid for all prime powers, including powers of $2$.
\end{remark}

\appendix
\section{A reproducible Jacobian certificate}\label{app:jacobian}

For completeness we describe the certification used in
Lemma~\ref{lem:jacobian}. At the point $(2,3,4,5,6)$ every radicand
$z_i z_j+\eps$ is a positive integer. Fix a decimal precision $P$ and use
integer square roots to find rational endpoints
\[
 \frac{k}{10^P}\le\sqrt m<\frac{k+1}{10^P}
\]
for each required integer $m$. All subsequent additions, multiplications,
divisions and derivatives are carried out with exact rational interval
arithmetic. The determinant is evaluated as the sum of its $120$ signed
permutation terms. With $P=90$ the program prints
\begin{verbatim}
sigma 1
 decimal enclosure: ('-0.000000000000000000000000000000122625271490038',
                     '-0.000000000000000000000000000000122625271490037')
 excludes zero: True
sigma -1
 decimal enclosure: ('0.000000000000000000000000000029335784988635907',
                     '0.000000000000000000000000000029335784988635908')
 excludes zero: True
\end{verbatim}
which are precisely the enclosures stated in Lemma~\ref{lem:jacobian}.
The ancillary file \path{jacobian_certificate.py}, available on the author's web page, uses only Python's
standard-library classes \texttt{Fraction} and \texttt{isqrt}, together with
integer arithmetic. It was verified with Python~3.13.5 and can be run by
\begin{verbatim}
python3 jacobian_certificate.py
\end{verbatim}
No floating-point assumption enters the nonvanishing assertion.

\section*{Acknowledgments}
The author acknowledges support from
the Croatian Science Foundation under the
project no. IP-2022-10-5008 (TEBAG),
the project ``Implementation of cutting-edge research and its application as part of the
Scientific Center of Excellence for Quantum and Complex Systems, and Representations of Lie Algebras'', Grant No. PK.1.1.10.0004, co-financed by the European
Union through the European Regional Development Fund -- Competitiveness and
Cohesion Programme 2021--2027,
and the European Union:
NextGenerationEU through the National Recovery and Resilience Plan 2021--2026, Institutional grant of University of Zagreb Faculty of Science (IK IA 1.1.3. Impact4Math).

The author acknowledges the use of ChatGPT 5.6 Sol (OpenAI) in this work.
In particular, discussions with ChatGPT concerning earlier manuscript versions of \cite{DL}
and attempts to prove nonvanishing of the eliminant led to the toric elimination lemma (Lemma \ref{lem:toric}).
The author is responsible for the proofs, verification of results and final presentation.

The author would like to thank Tomislav Pejkovi\' c and Matija Kazalicki for their useful comments and suggestions.

\end{document}